\documentclass[12pt,a4paper]{article}

\usepackage{mathptmx}       % selects Times Roman as basic font
\usepackage{helvet}         % selects Helvetica as sans-serif font
\usepackage{courier}        % selects Courier as typewriter font
\usepackage{type1cm}        % activate if the above 3 fonts are
\usepackage{makeidx}         % allows index generation
\usepackage{graphicx}        % standard LaTeX graphics tool
\usepackage{multicol}        % used for the two-column index
\usepackage[bottom]{footmisc}% places footnotes at page bottom

\makeindex             % used for the subject index
\usepackage{enumerate,bbm, amssymb, amsmath}
\newcommand{\bol}[1]{\mbox{\boldmath$#1$}}

\newcommand{\bSigma}{\bol{\Sigma}}
\newcommand{\bS}{\mathbf{S}}

\newcommand{\bI}{\mathbf{I}}

\newcommand{\bX}{\mathbf{X}}
\newcommand{\by}{\mathbf{y}}
\newcommand{\bY}{\mathbf{Y}}
\newcommand{\bB}{\mathbf{B}}
\newcommand{\bA}{\mathbf{A}}

\newcommand{\bzero}{\mathbf{0}}
\newtheorem{theorem}{Theorem}

\newtheorem{corollary}{Corollary}
\newenvironment{proof}{\paragraph{Proof:}}{\hfill$\square$}

\begin{document}

\title{Spectral analysis of large reflexive generalized inverse and Moore-Penrose inverse matrices}
%\titlerunning{Spectral analysis of inverse matrices}
\author{Taras Bodnar\\ Department of Mathematics, Stockholm University \and Nestor Parolya\\ Delft Institute of Applied Mathematics, Delft University of Technology}
%\institute{Taras Bodnar \at Department of Mathematics, Stockholm University, Stockholm, Sweden, \\\email{taras.bodnar@math.su.se}
%\and Nestor Parolya \at Delft Institute of Applied Mathematics, Delft University of Technology, Delft, The Netherlands, \email{n.parolya@tudelft.nl}}
\maketitle

\abstract{A reflexive generalized inverse and the Moore-Penrose inverse are often confused in statistical literature but in fact they have completely different behaviour in case the population covariance matrix is not a multiple of identity. In this paper, we study the spectral properties of a reflexive generalized inverse and of the Moore-Penrose inverse of the sample covariance matrix.  The obtained results are used to assess the difference in the asymptotic behaviour of their eigenvalues.
  %Furthermore, several applications are provided to the problems of the estimation of the precision matrix and functionals involving precision matrix in high dimensions.
}
\section{Introduction}

Let $\bY_n=(\by_1,\by_2,...,\by_n)$ be the $p \times n$ data matrix which consists of $n$ column vectors of dimension $p$ with $E(\by_i)=\bzero$ and $Cov(\by_i)=\bSigma$ for $i \in 1,...,n$. We assume that $p/n\rightarrow c\in (1, +\infty)$ as $n\rightarrow\infty$. This type of limiting behavior is also referred to a ''large dimensional asymptotics'' or ''the Kolmogorov asymptotics''. In this case, the traditional estimators perform very poorly and tend to over/underestimate the unknown parameters of the asset returns, i.e., the mean vector and the covariance matrix.

Throughout this paper it is assumed that there exists a $p\times n$ random matrix $\bX_n$ which consists of independent and identically distributed (i.i.d.) real random variables with zero mean and unit variance such that
\begin{equation}\label{obs}
 \bY_n=  \bSigma^{\frac{1}{2}}\bX_n \,,
 \end{equation}
where the matrix $\bSigma^{\frac{1}{2}}$ denotes the symmetric square root of the population covariance matrix $\bSigma$. Other square roots of $\bSigma$, like the lower triangular matrix of the Cholesky decomposition, can also be used. Note that the observation matrix $\bY_n$ consists of dependent rows although its columns are independent. It is worth mentioning that although the assumption of time independence looks quite restrictive in real-life applications, the model can be extended to dependent variables (for instance, causal AR(1) models) \footnote{Bai and Zhou ~\cite{BaiZhou2008} define a general dependence structure in the following way: for all $k$, $E(Y_{jk}Y_{lk})=\sigma_{lj}$, and for any non-random matrix $\bB$ with bounded norm, it holds that $E\left|\mathbf{y}_k^\top\bB\mathbf{y}_k-\mbox{tr}(\bB\bSigma)\right|^2=o\left(n^2\right)$ where $\bSigma=(\sigma_{lj})$.} by imposing more complicated conditions on the elements of $\bY_n$ (see, ~\cite{BaiZhou2008} for details) or by assuming an $m$-dependent structure (e.g., ~\cite{HuiPan2010} for a finite number of dependent entries, and ~\cite{Friesen2013} and ~\cite{Wei2016} for a possibly increasing number). Nevertheless, this will not change the main results of our paper and would only make the proofs more technical. That is why we assume independence for the sake of brevity and transparency.

%The main assumption which is used throughout the paper is
%\begin{description}
%\item {\bf (A1)} The noise matrix $\bX_n$ is a $p\times n$ matrix which consists of i.i.d. real random variables with zero mean and  unit variance.
%%\item {\bf (A2)} The covariance matrix $\bSigma$ is positive definite with minimum and minimum eigenvalues uniformly bounded away from zero and infinity, respectively.
%\end{description}

% These two regularity conditions are very general and they fit many practical situations. Assumption (A1) is common for financial and statistical problems. It does not impose a strong restriction on the data-generating process. Assumption (A2) is purely technical. Moreover, it seems to influence only the convergence rate of the proposed estimator (see, e.g., Rubio et al. (2012)).\\

The sample covariance matrix is given by (e.g., ~\cite{KollovonRosen2005})
\begin{equation}\label{samplecov}
 \bS_n=\frac{1}{n}\bY_n\bY_n^{\prime}=\frac{1}{n}\bSigma^{\frac{1}{2}}\bX_n\bX_n^{\prime}\bSigma^{\frac{1}{2}}.
\end{equation}
Throughout the paper it is assumed that the sample size $n$ is smaller than the dimension $p$ of random vectors $\by_i$, $i=1,...,n$, that is $p/n \to c >1$ as $n \to \infty$. In this case the sample covariance matrix is singular and its inverse does not exist (cf., \cite{BodnarGuptaParolya2016,BodnarOkhrin2008,ImorivonRosen2020}). On the other side, the inverse of the population covariance matrix $\bSigma$ is present in many applications from finance, signal processing, biostatistics, environmentrics, etc. (see, e.g., ~\cite{BodnarDmytrivParolyaSchmid2019,BodnarParolyaSchmid2018,Holgersson2020}). In practice, the Moore-Penrose inverse is usually employed (cf., ~\cite{BodnarDetteParolya2016,ImorivonRosen2020}), while the other types of the generalize inverse (see, ~\cite{Wang2018}) can also be used.

In this paper we compare the spectral properties of two generalized inverses of the singular sample covariance matrix given by:
\begin{itemize}
\item Moore-Penrose inverse:
\begin{equation}\label{S+}
\bS_n^{+}=\left(\frac{1}{n}\bY_n\bY_n^{\prime}\right)^+=\frac{1}{n}\bY_n\left(\frac{1}{n}\bY_n^{\prime}\bY_n\right)^{-2}\bY_n^{\prime}
=\frac{1}{n}\bSigma^{\frac{1}{2}}\bX_n\left(\frac{1}{n}\bX_n^{\prime}\bSigma\bX_n\right)^{-2}\bX_n^{\prime}\bSigma^{\frac{1}{2}}
\end{equation}

\item Reflexive inverse
\begin{equation}\label{S-}
\bS_n^{-}=\bSigma^{-\frac{1}{2}}\left(\frac{1}{n}\bX_n\bX_n^{\prime}\right)^+\bSigma^{-\frac{1}{2}}
=\frac{1}{n}\bSigma^{-\frac{1}{2}}\bX_n\left(\frac{1}{n}\bX_n^{\prime}\bX_n\right)^{-2}\bX_n^{\prime}\bSigma^{-\frac{1}{2}}
\end{equation}
\end{itemize}

Although the Moore-Penrose inverse $\bS^+$ can directly be computed from the observation matrix $\bY_n$, the derivation of its stochastic properties might be challenging in some practical problems. On the other side, the computation of the considered reflexive inverse $\bS^-$ is not possible in practice, but the derivation of the stochastic properties is considerably simplified. The goal of this paper is to quantify the difference between the two generalized inverse matrices with the aim to develop a statistical methodology to assess the estimation errors when the Moore-Penrose inverse is used.

The rest of the paper is structured as follows. In the next section, we study the asymptotic properties of $\bS^+$ and $\bS^-$ by deriving two integral equations whose solutions are the Stieltjes transforms of the limiting spectral distributions of $\bS^+$ and $\bS^-$. These findings are used in Section 3, where the asymptotic behaviour of the Frobenius norm of the difference between $\bS^+$ and $\bS^-$ is investigated. Section 4 presents the results of a numerical illustration, while concluding remarks are provided in Section 5.

\section{Asymptotic properties of $\bS^+$ and $\bS^-$}

For a symmetric matrix $\bA$ we denote by $\lambda_1(\bA)\geq\ldots\geq\lambda_p(\bA)$ its ordered eigenvalues and  by  $F^{\bA}(t)$ the corresponding empirical distribution function (e.d.f.), that is
 $$
 F^{\bA}(t) =   \frac{1}{p}\sum\limits_{i=1}^{p}\mathbbm{1}{  \{  \lambda_i(\bA)  \leq t\}}  ,
 $$
where $\mathbbm{1}{\{\cdot\}}$ is the indicator function. Furthermore, for a function $G: \mathbbm{R} \to \mathbbm{R} $ of bounded variation the Stieltjes transform is introduced by
$$
  m_G(z)=\int\limits_{-\infty}^{+\infty}\frac{1}{\lambda-z}dG(\lambda); ~~~z\in\mathbbm{C}^+\equiv\{z\in\mathbbm{C}: \Im z>0\} \,.
$$
Note that there is a direct connection between $m_G(z)$ and moment generating function of $G$, $\Psi_G(z)$ given by
\[\Psi_G(z)=-\frac{1}{z}m_G\left(\frac{1}{z}\right)-1.\]

In Theorem \ref{th1} we present the expressions of the Stieltjes transform for $\bS^+$ and $\bS^-$.

\begin{theorem}\label{th1}
Let the $p\times n$ noise matrix $\bX_n$ consist of i.i.d. real random variables with zero mean and unit variance. Assume that $\bSigma_n$ is nonrandom symmetric and positive definite with a bounded spectral norm and the e.d.f. $H_n=F^{\bSigma_n}$ converges weakly to a nonrandom distribution function $H$. Then
\begin{enumerate}[(i)]
\item the e.d.f. $F^{\bS_n^+}$ converges weakly almost surely to some deterministic c.d.f.
  $P^+$ whose Stieltjes transformation $m_P$ satisfies the following equation:
$$
  m_{P^+}(z) = -\frac{1}{z}\Bigg (2- c^{-1} + \int_{-\infty}^{+\infty}\frac{dH(\tau)}{z\tau c(zm_{P^+}(z)+1)-1}\Bigg);
$$
\item the e.d.f. $F^{\bS_n^-}$ converges weakly almost surely to some deterministic c.d.f.
  $P^-$ whose Stieltjes transformation $m_{P^-}$ satisfies the following equation:
$$
  m_{P^-}(z)= -\frac{1}{z}-\frac{1}{z}\int\limits_{-\infty}^{+\infty}\frac{dH(\tau)}{\tau cz^2m_{P^-}(z)\left(1-\frac{c}{1-c-czm_{P^-}(z)}\right) -1}  \,.
$$
\end{enumerate}
\end{theorem}

\begin{proof}
\begin{enumerate}[(i)]
\item The result of part (i) of the theorem was derived in Theorem 2.1 of ~\cite{BodnarDetteParolya2016}.
\item In order to prove the asymptotic result in the case of $\bS^-_n$, whose eigenvalues are the same as those of the product of two matrices $\bSigma^{-1}$ and $\tilde{\bS}_n^+=1/n\bX_n\left(\frac{1}{n}\bX_n^{\prime}\bX_n\right)^{-2}\bX_n^{\prime}$, we use subordination results of free probability (see, e.g., ~\cite{BB2007}). Namely for two Borel probability measures $\mu$ and $\nu$ on $[0, +\infty)$ (both are limits of $F^{\tilde{S}_n^+}$ and $F^{\bSigma^{-1}}$, respectively), there exist two unique analytic functions $F_1, F_2: \mathbbm{C}\setminus[0, +\infty)\rightarrow \mathbbm{C}\setminus[0, +\infty)$ such that
  \begin{eqnarray}\label{freeprob}
\frac{F_1(z)F_2(z)}{z}=\eta_{\mu}(F_1(z))=\eta_{\nu}(F_2(z))=\eta(z)\,,
  \end{eqnarray}
where $\eta(z)=1-\frac{z}{g_{P^-}(1/z)}$ is the eta transform and $g_{P^-}$ is the Cauchy transform (the negative Stieltjes transform), i.e., $g_{P^-}(z)=-m_{P^-}(z)$, for the limit of $F^{\bS^-_n}$ denoted by $P^-$. In particular, from the first and the second equalities in \eqref{freeprob} we get that for any $z\in\mathbbm{C}^+=\{z: \Im(z)>0\}$ such that $F_2$ is analytic at $z$, the function $F_2(z)$ satisfies the following equation:\footnote{In fact, a similar equation also holds for $F_1$ when one replaces $\mu$ and $\nu$ in \eqref{subordinator} but for our purposes it is enough to know one of them.}
    \begin{eqnarray}\label{subordinator}
      \eta_\nu(F_2(z))=\eta_\mu\left(\frac{\eta_\nu(F_2(z))z}{F_2(z)}\right)\,.
    \end{eqnarray}
On the other hand, from the last equality in \eqref{freeprob} we have that $g_{P^-}(z)$ satisfies the equality
    \begin{eqnarray}\label{cauchy}
      g_{P^-}(z)=\frac{1}{zF_2(1/z)}g_\nu\left(\frac{1}{F_2(1/z)}\right)\,.
    \end{eqnarray}
Thus, we need first to find the so called subordination function $F_2(z)$ from \eqref{subordinator} and plug it into \eqref{cauchy}. For simplicity we further suppress the subindex of $g_{P^-}$.

Let
    \begin{eqnarray}
      \Theta(z)=\frac{F_2(z)}{z\eta(z)}
    \end{eqnarray}
and rewrite \eqref{subordinator} using $\eta(z)=\eta_\nu(F_2(z))$ in the following way
    \begin{eqnarray}
 \eta(z)=   \eta_\mu\left( \frac{\eta(z)z}{F_2(z)}\right)=\eta_\mu\left(\frac{1}{\Theta(z)}\right)\,.
    \end{eqnarray}

Using the definition of the eta transform we get
    \begin{eqnarray*}
      1-\frac{1}{\frac{1}{z}g(1/z)}=1-\frac{1}{\Theta(z)}\frac{1}{g_\mu(\Theta(z))}
    \end{eqnarray*}
and, hence,
    \begin{eqnarray}\label{submod}
     \frac{1}{z}g(1/z)=\Theta(z)g_\mu(\Theta(z))\,.
    \end{eqnarray}

From ~\cite{BodnarDetteParolya2016} we get that
\begin{eqnarray}
  g_\mu(z)=-m_{P^-}(z)=\frac{1}{z}\left(2-c^{-1}+\frac{m_{MP}(1/z)}{z} \right)\,,
\end{eqnarray}
where $m_{MP}(z)$ is the Stieltjes transformation of the Marchenko-Pastur law given by (see, e.g., ~\cite{BaiSilverstein2010})
\begin{eqnarray}
  m_{MP}(z)=\frac{1}{2cz}\left(1-c-z+\sqrt{(1+c-z)^2-4c}\right)\,.
\end{eqnarray}
Thus, the equation \eqref{submod} becomes
\begin{eqnarray*}
  \frac{1}{z}g(1/z)&=&\Theta(z)\frac{1}{\Theta(z)}\left(2-c^{-1}+\frac{m_{MP}(1/\Theta(z))}{\Theta(z)} \right)\\
                   &=& 2-c^{-1}+\frac{1-c-\frac{1}{\Theta(z)}+\sqrt{(1+c-\frac{1}{\Theta(z)})^2-4c}}{\Theta(z)2c\frac{1}{\Theta(z)}}\\
  &=&\frac{1}{2c}\left(2c-2+(1+c-\frac{1}{\Theta(z)})+\sqrt{(1+c-\frac{1}{\Theta(z)})^2-4c} \right)\,,
\end{eqnarray*}
or, equivalently, by rearranging terms we obtain
\begin{eqnarray}
  2\left(c(\frac{1}{z}g(1/z)-1)+1\right)-(1+c-\frac{1}{\Theta(z)})=\sqrt{(1+c-\frac{1}{\Theta(z)})^2-4c}\,,
\end{eqnarray}
where by squaring of both sides, we get
\[\left(c(\frac{1}{z}g(1/z)-1)+1\right)^2-\left(c(\frac{1}{z}g(1/z)-1)+1\right)(1+c-\frac{1}{\Theta(z)})+c=0\]
or,
 % &&\left(c(\frac{1}{z}g(1/z)-1)+1\right)^2-c(\frac{1}{z}g(1/z)-1)(1+c-\frac{1}{\Theta(z)})-1-c+\frac{1}{\Theta(z)}+c=0\\
\begin{eqnarray*}
  &&1+c^2(\frac{1}{z}g(1/z)-1)^2+2c(\frac{1}{z}g(1/z)-1)-1+\frac{1}{\Theta(z)}\\
  &-&c(\frac{1}{z}g(1/z)-1)(1+c-\frac{1}{\Theta(z)})=0,\\
\end{eqnarray*}
which yields
\begin{equation}\label{eq1_proof_th1}
 c(\frac{1}{z}g(1/z)-1)+(1-c)+\frac{1}{\Theta(z)c(\frac{1}{z}g(1/z)-1)}+\frac{1}{\Theta(z)}=0\,.
\end{equation}

From \eqref{eq1_proof_th1} we find $\Theta(z)$ as a function of $g(1/z)$ expressed as
\begin{eqnarray}
  \Theta(z)=\frac{1+c(\frac{1}{z}g(1/z)-1)}{c(\frac{1}{z}g(1/z)-1)(c-1-c(\frac{1}{z}g(1/z)-1))}
\end{eqnarray}
or, in terms of $F_2(z)$ given by
\begin{eqnarray}
  F_2(z)&=&\frac{z-cz+cg(1/z)}{c(\frac{1}{z}g(1/z)-1)(c-1-c(\frac{1}{z}g(1/z)-1))}\frac{\frac{1}{z}g(1/z)-1}{\frac{1}{z}g(1/z)}\nonumber\\
        &=&\frac{z-cz+cg(1/z)}{c(-z+2cz-cg(1/z))}\frac{z^2}{g(1/z)}\,.
 % &=&-\frac{z^2}{g(1/z)}\left(c^{-1}+\frac{z}{z-2cz+cg(1/z)}\right)
\end{eqnarray}
At last, we use the property $g_{\nu}(z)=1/z-1/z^2g_H(1/z)$ and plug $F_2(z)$ in \eqref{cauchy}. This leads to
\begin{eqnarray*}
  g(z)&=&\frac{1}{zF_2(1/z)}\left(F_2(1/z)-F^2_2(1/z)g_H(F_2(1/z))\right)\\
            &=&\frac{1}{z}-\frac{F_2(1/z)}{z}g_H(F_2(1/z))=\frac{1}{z}-\frac{F_2(1/z)}{z}\int\limits_{-\infty}^{+\infty}\frac{dH(\tau)}{F_2(1/z)-\tau}\\
  &=&  \frac{1}{z}+\int\limits_{-\infty}^{+\infty}\frac{dH(\tau)}{\tau z/F_2(1/z)-z}=\frac{1}{z}+\int\limits_{-\infty}^{+\infty}\frac{dH(\tau)}{-\tau cz^3g(z)\left(1-\frac{c}{1-c+czg(z)}\right) -z}\,.
\end{eqnarray*}
The latter can be rewritten in terms of Stieltjes transform $m_{P^-}(z)$ by
\begin{eqnarray}
  m_{P^-}(z)=-\frac{1}{z}-\frac{1}{z}\int\limits_{-\infty}^{+\infty}\frac{dH(\tau)}{\tau cz^2m_{P^-}(z)\left(1-\frac{c}{1-c-czm_{P^-}(z)}\right) -1}\,
\end{eqnarray}
and the second part of theorem is proved.
  \end{enumerate}
 \end{proof}

Although the Stieltjes transforms of two generalize inverse matrices are quite different, they have one in common: besides the fact that they are equal for $\bSigma=\bI$, we also observe that they become close to each other in case $c$ tends to one from the right. Thus, if $p/n$ is close to one from the right, then there should be no large difference in using the Moore-Penrose inverse $\bS^+$ or the generalized reflexive inverse $\bS^-$ asymptotically.

\section{Quantification the difference between the asymptotic behaviour of $\bS^+$ and $\bS^-$}
Using the Stieltjes transforms computed in the case of $\bS^+$ and $\bS^-$, we are able to compare their moments. In order to quantify this difference more carefully we consider the quadratic or the so-called Frobenius norm of the difference between $\bS^+$ and $\bS^-$ given by
\begin{eqnarray*}
\|\bS^--\bS^+\|^2_F&=&\mbox{tr}\left((\bS^--\bS^+)(\bS^--\bS^+)^\top\right)\\
&=&\mbox{tr}\left(\bS^-\bS^-\right)-2\mbox{tr}\left(\bS^+\bS^-\right)+\mbox{tr}\left(\bS^+\bS^+\right)\\
&=&\mbox{tr}\left(\bS^-\bS^-\right)-\mbox{tr}\left(\bS^+\bS^+\right)\\
&=&\|\bS^-\|^2_F-\|\bS^+\|^2_F\,.
\end{eqnarray*}

In Theorem \ref{frob} the asymptotic equivalents for both Frobenius norms are provided.

\begin{theorem}\label{frob} Under assumptions of Theorem \ref{th1} it holds almost surely as $p/n\to c>1$
  {
  \begin{eqnarray}
   && 1/p \|\bS^+\|^2_F \longrightarrow c^{-1}\left(\frac{1}{m^2_{\underline{F}}(0)}-c\int\limits_{-\infty}^{+\infty}\frac{\tau^2 dH(\tau)}{(1+\tau m_{\underline{F}}(0))^2 }\right)^{-1},\\
    && 1/p \|\bS^-\|^2_F \longrightarrow \frac{(1+c(c-1))}{c^2(c-1)^3}\left(\int\limits_{-\infty}^{+\infty}\frac{dH(\tau)}{\tau} \right)^2+\frac{1}{(c(c-1))^2}\int\limits_{-\infty}^{+\infty}\frac{dH(\tau)}{\tau^2}\,,
  \end{eqnarray}
} where $m_{\underline{F}}(0)$ satisfies the following equation
\begin{eqnarray*}
       \frac{1}{m_{\underline{F}}(0)}=c\int\limits_{-\infty}^{+\infty}\frac{\tau dH(\tau)}{1+\tau m_{\underline{F}}(0)}\,.
\end{eqnarray*}
\end{theorem}

\begin{proof} In order to prove the statement of the theorem, we rewrite the Frobenius norm of random matrix $\bA$ for all $z\in\mathbbm{C}^+$ in the following way
  \begin{eqnarray}\label{frob_norm}
    1/p \|\bA\|^2_F=1/p \text{tr}(\bA^2)=-\left.\frac{1}{2}\frac{\partial^2}{\partial z^2}\frac{1}{p}\frac{\text{tr}(\bA-1/z\bI)^{-1}}{z}\right|_{z=0}
    =-\left.\frac{1}{2}\frac{\partial^2}{\partial z^2}\frac{m_{F^{\bA}}(1/z)}{z}\right|_{z=0}\,.
  \end{eqnarray}
  
We start with $\bA=\bS^+$. Let
  \[
  \Gamma_n(z)=\frac{m_{F^{\bS^+}}(1/z)}{z}\]
  Due to Theorem \ref{th1}.i we get that $\Gamma^+_n(z)$ converges almost surely to deterministic $\Gamma^+(z)$, which satisfies the following asymptotic equation
  \begin{eqnarray}\label{Gamma_plus}
 \Gamma^+(z) = -\Bigg (2- c^{-1} + \int_{-\infty}^{+\infty}\frac{zdH(\tau)}{\tau c(\Gamma^+(z)+1)-z}\Bigg)\,.
  \end{eqnarray}

For the computation of $1/p \|\bS^+\|^2_F$, the quantities $\Gamma^+(z)$, $\frac{\partial }{\partial z}\Gamma^+(z)$ and $\frac{\partial^2 }{\partial z^2}\Gamma^+(z)$ should be evaluated at zero. We rewrite \eqref{Gamma_plus} in an equivalent way
  \begin{eqnarray}\label{simple}
    \Gamma^+(z)=-1-c^{-1}zm_{\underline{F}}(z),
  \end{eqnarray}
  where $m_{\underline{F}}(z)$ is the limiting Stieltjes transform of $m_{F^{1/n\bY^\top\bY}}$, which satisfies the following equation (see, e.g., ~\cite{SilversteinBai1995} 
  \begin{eqnarray}\label{bai_sil}
    m_{\underline{F}}(z)=\left(c\int\limits_{-\infty}^{+\infty}\frac{\tau dH(\tau)}{1+\tau m_{\underline{F}}(z)}-z\right)^{-1}\,.
  \end{eqnarray}

The advantage of \eqref{simple} over \eqref{Gamma_plus} lies in the calculation of the derivatives, which can be done easier by the former. Since the quantity $m_{\underline{F}}(z)$ is bounded at zero for $c>1$, we immediately obtain
  \begin{eqnarray}
    \Gamma^+(0)\equiv \lim\limits_{z\to0^+}\Gamma^+(z)=-1\,.
  \end{eqnarray}

The first derivative of $\Gamma^+(z)$ is given by
  \begin{eqnarray}\label{first_der_gamma_plus}
 \frac{\partial}{\partial z}   \Gamma^+(z)=-c^{-1}m_{\underline{F}}(z)-c^{-1}z\frac{\partial}{\partial z}m_{\underline{F}}(z)\,
  \end{eqnarray}
and, thus, taking the limit $z\to0^+$ and using that $\frac{\partial}{\partial z}m_{\underline{F}}(z)=O(1)$ as $z\to0^+$, we get
  \begin{eqnarray}
  \Gamma^{'\;+}(0)\equiv\lim\limits_{z\to0^+}\frac{\partial}{\partial z}\Gamma^+(z)=-c^{-1}m_{\underline{F}}(0)\,,
  \end{eqnarray}
where $m_{\underline{F}}(0)$ satisfies the equation
  \begin{eqnarray}\label{bai_sil_zero}
       \frac{1}{m_{\underline{F}}(0)}=c\int\limits_{-\infty}^{+\infty}\frac{\tau dH(\tau)}{1+\tau m_{\underline{F}}(0)}\,.
  \end{eqnarray}

The second derivative of $\Gamma^+(z)$ is equal to
  \begin{eqnarray}\label{sec_der_Gamma_plus}
    \frac{\partial^2}{\partial z^2}   \Gamma^+(z)&=& -2c^{-1}\frac{\partial}{\partial z}m_{\underline{F}}(z)- c^{-1}z\frac{\partial^2}{\partial z^2}m_{\underline{F}}(z).
%    &=&-2c^{-1}\frac{\partial}{\partial z}m_{\underline{F}}(z)+O(z)\,.
  \end{eqnarray}
Denoting $m'_{\underline{F}}(z)=\frac{\partial}{\partial z}m_{\underline{F}}(z)$ and using \eqref{bai_sil} together with \eqref{bai_sil_zero}, we obtain
  \begin{eqnarray*}
 m'_{\underline{F}}(0)\equiv \lim\limits_{z\to0^+}m'_{\underline{F}}(z)  =\lim\limits_{z\to0^+}\frac{c\int\limits_{-\infty}^{+\infty}\frac{\tau^2 m'_{\underline{F}}(z)dH(\tau)}{(1+\tau m_{\underline{F}}(z))^2 }+1}{\left(c\int\limits_{-\infty}^{+\infty}\frac{\tau dH(\tau)}{1+\tau m_{\underline{F}}(z)} \right)^2}=\frac{ m'_{\underline{F}}(0)c\int\limits_{-\infty}^{+\infty}\frac{\tau^2 dH(\tau)}{(1+\tau m_{\underline{F}}(0))^2 }+1}{ \frac{1}{m^2_{\underline{F}}(0)}}
  \end{eqnarray*}
  or, equivalently,
  \begin{eqnarray}\label{2q1_th2}
    m'_{\underline{F}}(0)=\left(\frac{1}{m^2_{\underline{F}}(0)}-c\int\limits_{-\infty}^{+\infty}\frac{\tau^2 dH(\tau)}{(1+\tau m_{\underline{F}}(0))^2 }  \right)^{-1}\,.
  \end{eqnarray}

Finally, the application of \eqref{sec_der_Gamma_plus} and \eqref{2q1_th2} together with $\frac{\partial^2}{\partial z^2}m_{\underline{F}}(z)=O(1)$ as $z\to0^+$ leads to
  \begin{eqnarray*}
    &&\Gamma^{''\;+}(0)\equiv \lim\limits_{z\to0^+}  \frac{\partial^2}{\partial z^2}   \Gamma^+(z)=-2\frac{c^{-1}}{\frac{1}{m^2_{\underline{F}}(0)}-c\int\limits_{-\infty}^{+\infty}\frac{\tau^2 dH(\tau)}{(1+\tau m_{\underline{F}}(0))^2 }}\,.
       %= -2\frac{m^2_{\underline{F}}(0)}{c-\int\limits_{-\infty}^{+\infty}\frac{\tau^2m^2_{\underline{F}}(0) dH(\tau)}{(1+\tau m_{\underline{F}}(0))^2 }}\\
    %&=&-2\frac{m^2_{\underline{F}}(0)}{c+\int\limits_{-\infty}^{+\infty}\frac{\tau m_{\underline{F}}(0) dH(\tau)}{1+\tau m_{\underline{F}}(0) }-\int\limits_{-\infty}^{+\infty}\frac{dH(\tau)}{1+\tau m_{\underline{F}}(0) }+\int\limits_{-\infty}^{+\infty}\frac{dH(\tau)}{(1+\tau m_{\underline{F}}(0))^2 }}\\
    %&=&-2\frac{m^2_{\underline{F}}(0)}{c+c^{-1}-\int\limits_{-\infty}^{+\infty}\frac{dH(\tau)}{1+\tau m_{\underline{F}}(0) }+\int\limits_{-\infty}^{+\infty}\frac{dH(\tau)}{(1+\tau m_{\underline{F}}(0))^2 }}\,.
  \end{eqnarray*}
  Now, the first result of the theorem follows from \eqref{frob_norm}. 
  
Similarly, for the second identity of Theorem \ref{frob} we denote $\Gamma^-(z)$ as a limit of
  \begin{eqnarray}
    \Gamma^-_n(z)=\frac{m_{F^{\bS^-}}(1/z)}{z}.
  \end{eqnarray}
Then using Theorem \ref{th1}, the limiting function $\Gamma^-(z)$ satisfies the following asymptotic equation
  \begin{eqnarray}
    \Gamma^-(z)=-1-\int\limits_{-\infty}^{+\infty}\frac{zdH(\tau)}{\tau c\Gamma^-(z)\left(1-\frac{c}{1-c-c\Gamma^-(z)}\right) -z}\,.
  \end{eqnarray}
Here we immediately get that $\Gamma^-(z)$ is bounded at zero and it holds that
  \begin{eqnarray}
 &&  \lim\limits_{z\to0^+}\Gamma^-(z)=-1,\\
 && \lim\limits_{z\to0^+}\frac{\partial}{\partial z} \Gamma^-(z)=-\frac{1}{c(c-1)}\int\limits_{-\infty}^{+\infty}\frac{dH(\tau)}{\tau}\,,
  \end{eqnarray}
  where, the result for the first derivative of $\Gamma^-(z)$ follows from the identity
  \begin{eqnarray*}
  &&\left(1-z\int\limits_{-\infty}^{+\infty}\frac{\tau c\left(1- \frac{c(1-c)}{(1-c-c\Gamma^-(z))^2}\right)dH(\tau)}{(\tau c\Gamma^-(z)\left(1-\frac{c}{1-c-c\Gamma^-(z)}\right) -z)^2} \right)\frac{\partial}{\partial z} \Gamma^-(z)\\
    &=& -\int\limits_{-\infty}^{+\infty}\frac{\tau c\Gamma^-(z)\left(1-\frac{c}{1-c-c\Gamma^-(z)}\right)dH(\tau)}{(\tau c\Gamma^-(z)\left(1-\frac{c}{1-c-c\Gamma^-(z)}\right) -z)^2}\\
 %   &=&-\int\limits_{-\infty}^{+\infty}\frac{dH(\tau)}{(\tau c\Gamma^-(z)\left(1-\frac{c}{1-c-c\Gamma^-(z)}\right) -z)}-\int\limits_{-\infty}^{+\infty}\frac{zdH(\tau)}{(\tau c\Gamma^-(z)\left(1-\frac{c}{1-c-c\Gamma^-(z)}\right) -z)^2}\,.
    &=& -\int\limits_{-\infty}^{+\infty}\frac{dH(\tau)}{(\tau c\Gamma^-(z)\left(1-\frac{c}{1-c-c\Gamma^-(z)}\right) -z)}-\int\limits_{-\infty}^{+\infty}\frac{zdH(\tau)}{(\tau c\Gamma^-(z)\left(1-\frac{c}{1-c-c\Gamma^-(z)}\right) -z)^2}\,.
  \end{eqnarray*}
  
For the second derivative we get the following equality
  \begin{eqnarray*}
    (1+O(z))\frac{\partial^2}{\partial z^2} \Gamma^-(z)&=&2\frac{\partial}{\partial z} \Gamma^-(z)\int\limits_{-\infty}^{+\infty}\frac{\tau c\left(1- \frac{c(1-c)}{(1-c-c\Gamma^-(z))^2}\right)dH(\tau)}{(\tau c\Gamma^-(z)\left(1-\frac{c}{1-c-c\Gamma^-(z)}\right) -z)^2}\\
                                                       % &+& \frac{\partial}{\partial z} \Gamma^-(z)\left(\int\limits_{-\infty}^{+\infty}\frac{\tau c\left(1- \frac{c(1-c)}{(1-c-c\Gamma^-(z))^2}\right)dH(\tau)}{(\tau c\Gamma^-(z)\left(1-\frac{c}{1-c-c\Gamma^-(z)}\right) -z)^2}\right)\\
                     &-&             2 \int\limits_{-\infty}^{+\infty}\frac{dH(\tau)}{(\tau c\Gamma^-(z)\left(1-\frac{c}{1-c-c\Gamma^-(z)}\right) -z)^2}+O(z)\,
  \end{eqnarray*}
  and taking the limit $z\to0^+$ from both sides leads to
  \begin{eqnarray}\label{Gamma_minus_secprime}
    \lim\limits_{z\to0^+}\frac{\partial^2}{\partial z^2} \Gamma^-(z)&=&2\frac{\partial}{\partial z} \Gamma^-(z)\frac{c(1-c(1-c))}{c^2(c-1)^2}\int\limits_{-\infty}^{+\infty}\frac{dH(\tau)}{\tau}-2\frac{1}{(c(c-1))^2}\int\limits_{-\infty}^{+\infty}\frac{dH(\tau)}{\tau^2}\nonumber\\
    &=& -2\frac{(1+c(c-1))}{c^2(c-1)^3}\left(\int\limits_{-\infty}^{+\infty}\frac{dH(\tau)}{\tau} \right)^2-2\frac{1}{(c(c-1))^2}\int\limits_{-\infty}^{+\infty}\frac{dH(\tau)}{\tau^2}\,.
  \end{eqnarray}
The second statement of the theorem follows now from \eqref{Gamma_minus_secprime} and \eqref{frob_norm}.
\end{proof}

For a better visualization of the results of Theorem \ref{frob} from the statistical point of view, we present its empirical counterpart. Here, the almost sure convergence of the asymptotic equivalents are summarize in Corollary \ref{emp_frob}.

\begin{corollary}\label{emp_frob} Under assumptions of Theorem \ref{th1} for $p/n\to c>1$ holds almost surely
  \begin{eqnarray}
  && 1/p\left|  \|\bS^+\|^2_F - c^{-1}\left(\frac{p}{m^2_{\underline{F}}(0)}-c||(\bSigma+m_{\underline{F}}(0)\bI)^{-1} ||^2_F  \right)^{-1}   \right| \to 0\label{emp_frob_splus}\\
  && 1/p\left|   \|\bS^-\|^2_F - \frac{(1+c(c-1))}{c^2(c-1)^3}\frac{1}{p}\left(\text{tr}(\bSigma^{-1})\right)^2-\frac{1}{(c(c-1))^2}||\bSigma^{-1}||_F^2  \right| \to 0\label{emp_frob_sminus}
  \end{eqnarray}
where $m_{\underline{F}}(0)$ satisfies asymptotically (approximately) the following equation
\begin{eqnarray*}
    \frac{p}{m_{\underline{F}}(0)}= c\cdot\text{tr}(\bSigma+m_{\underline{F}}(0)\bI)^{-1}\,.
\end{eqnarray*}
\end{corollary}

The results of Corollary \ref{emp_frob} show how the Frobenius norms of both inverses are connected with their population counterpart $||\bSigma^{-1}||^2_F$. Namely, looking on the asymptotic behavior of $\|\bS^+\|^2_F$ one can easily deduce that it is not possible to estimate $||\bSigma^{-1}||^2_F$ consistently using the Moore-Penrose inverse because of the nonlinearity which is present in $(\bSigma+m_{\underline{F}}(0)\bI)^{-1}$. On the other side, it is doable by reflexive generalized inverse $\bS_n^-$. Indeed, using the proof of Theorem \ref{frob} one can find that
$$1/p\left|\text{tr}(\bS_n^-)- \frac{1}{c(c-1)}\text{tr}(\bSigma^{-1})\right|\to 0\,, $$
which together with \eqref{emp_frob_sminus} implies that
\begin{eqnarray}
  1/p\left| (c(c-1))^2  \left[\|\bS^-\|^2_F - \left(\frac{1}{(c-1)}+c\right)\frac{1}{p}\left(\text{tr}(\bS^{-})\right)^2\right]-||\bSigma^{-1}||_F^2  \right| \to 0 \,
\end{eqnarray}
almost surely.

\section{Numerical illustration}
Using the results of the previous section, we present several numerical results to quantify the difference $\|\bS^--\bS^+\|^2_F=\|\bS^-\|^2_F-\|\bS^+\|^2_F$ for some $\bSigma$. In order to avoid the normalization $1/p$ we will use the normalized Frobenius loss (NFL) expressed as
\begin{equation}\label{NFL}
 NFL=\frac{\|\bS^--\bS^+\|^2_F}{\|\bS^+\|^2_F} =\frac{\|\bS^-\|^2_F}{\|\bS^+\|^2_F}-1,
\end{equation}
which measures the difference between $\bS^-$ and $\bS^+$ normalized by the Frobenius norm of the Moore-Penrose inverse. The application of \eqref{NFL} with $\bS^+$ and $\bS^-$ as in \eqref{S+} and \eqref{S-} leads to the so called empirical NFL, while the usage of the results of Theorem 2 corresponds to the asymptotic NFL. The latter can be interpreted how much both Frobenius norms differ asymptotically.

In Figure \ref{fig1} we present the results of a simulation study where the normalized Frobenius losses are computed for several values of the concentration ratio $c>1$ as a function of dimension $p$. For the sake of illustration, we provide the results obtained for the samples generated from the multivariate normal distribution, while similar values were also obtained for other multivariate distributions. We set the mean vector to zero and, without loss of generality, use the diagonal covariance matrix $\bSigma$ with 20\% of eigenvalues equal to one, 40\% equal to three and the rest equal to ten. The results for $c=1.07$ confirm our expectations discussed after the proof of Theorem 1, namely there is no large difference between both norms: NFL is small and the empirical NFL converges quite fast to its asymptotic counterpart. By increasing $c$ we observe that the NFL increases indicating that both norms deviate from each other. For example, in case $c=2$ the asymptotic and empirical NFLs are close to 1.2 what means that the Frobenius norm for reflexive inverse is more than a double of the Frobenius norm of the Moore-Penrose inverse. This observation becomes more severe for larger values of $c$. Moreover, for $c=10$ we observe a bit slower convergence of the sample NFL to its asymptotic value.

\begin{figure}
  \centering
  \begin{tabular}{c}
 \includegraphics[scale=0.28]{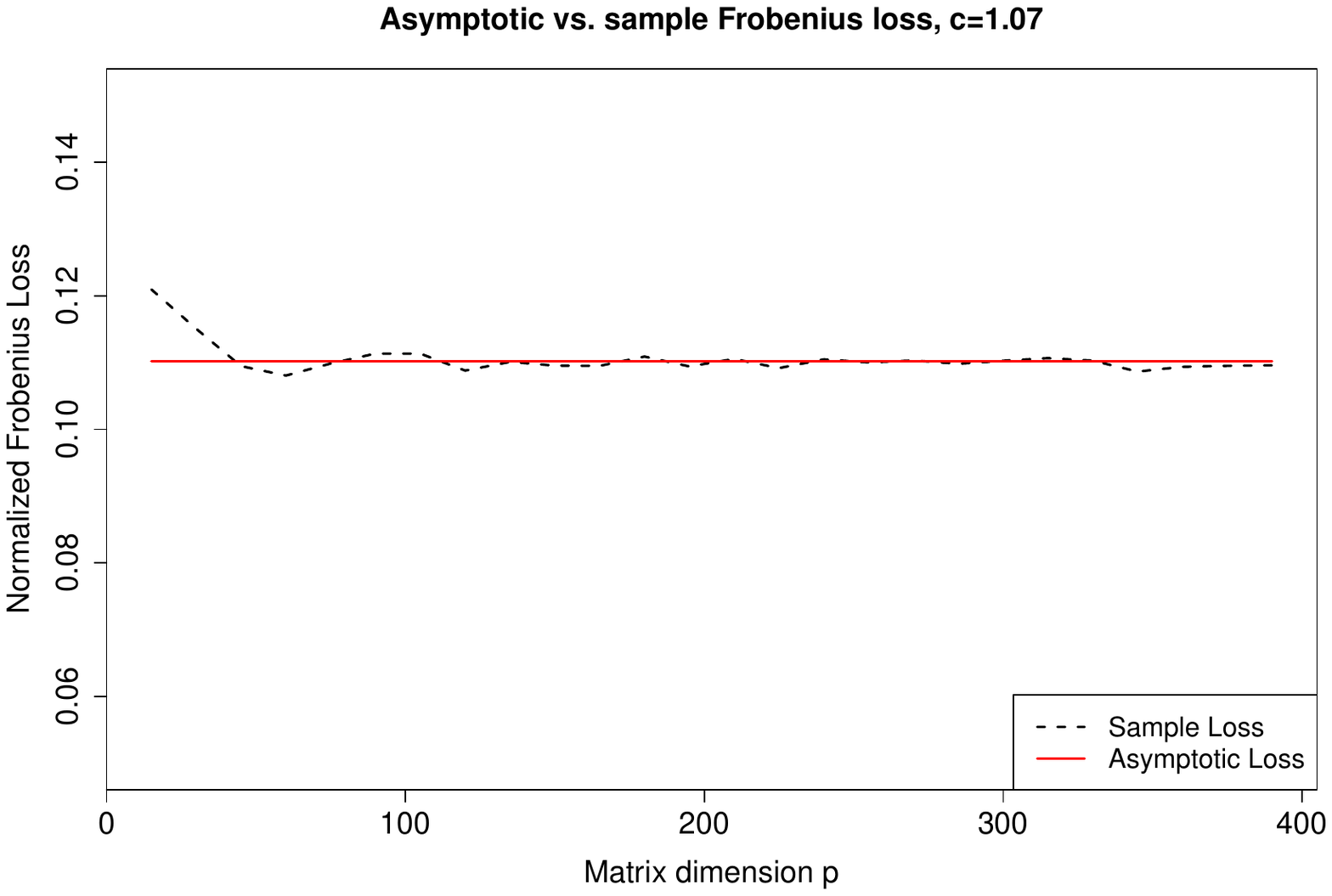}\\  \includegraphics[scale=0.28]{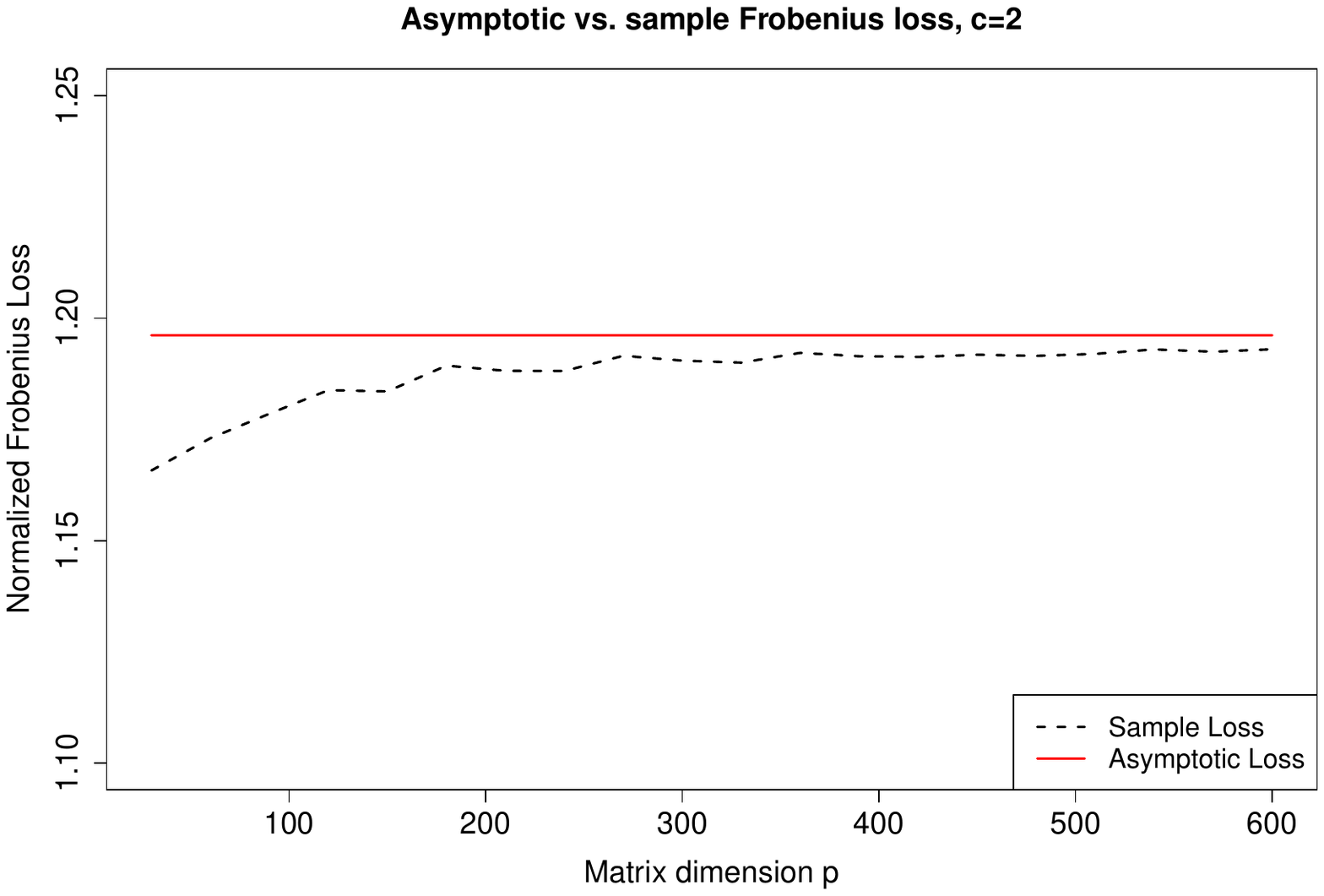}\\
    \includegraphics[scale=0.28]{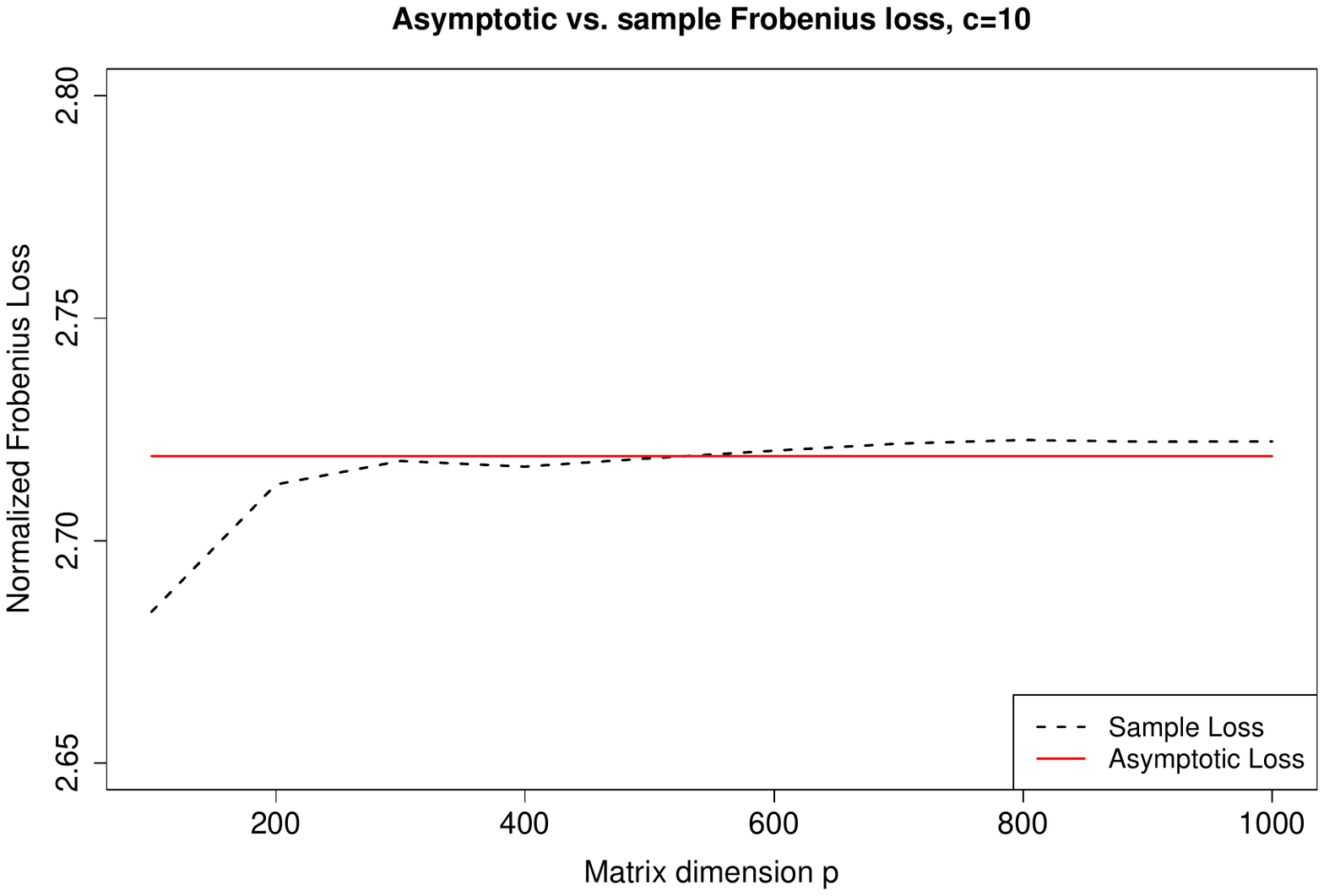}
  \end{tabular}

  \caption{Normalized Frobenius losses $\frac{\|\bS^--\bS^+\|^2_F}{\|\bS^+\|^2_F}$ for several values of $c>1$ as function of $p$.}\label{fig1}
\end{figure}

The obtained findings indicate that the usage of $\bS^+_n$ instead of $\bS_n^-$ must be done with much care and are only reliable if the concentration ratio $p/n$ is close to one. Otherwise, one can not expect neither a good approximation of functionals depending on the inverse covariance matrix nor consistent estimators for them.

\section{Summary}

In many statistical applications the dimension of the data-generating process is larger than the sample size. However, one still needs to compute the inverse of the sample covariance matrix, which is present in many expressions. There are a plenty of ways how to define a generalized inverse with the Moore-Penrose inverse and reflexive generalized inverse matrices be the mostly used ones. While the former is easily computable and is unique, the latter can not be calculated from the data. On the other side, the behavior of the Moore-Penrose inverse in many high-dimensional applications is far away from a satisfactory one, whereas the reflexive inverse obeys very convenient statistical and asymptotic properties. Namely, the application of the reflexive inverse allows to estimate functionals of the precision matrix (the inverse of population covariance matrix) consistently in high-dimension. 

In this work we study spectral properties of both inverses in detail. The almost sure limits of the Stieltjes transforms of their empirical distribution functions of eigenvalues (so-called limiting spectral distributions) is provided in the case the dimension $p$ and the sample size $n$ increase to infinity simultaneously such that their ratio $p/n$ tends to a positive constant $c$ greater than one. We discover that both limiting spectral distributions differ considerably and only coincide when $c$ tends to one. The results are illustrated via the calculation of asymptotic normalized Frobenius loss of both inverse matrices. Finally, we justify the obtained theoretical findings via a simulation study.

\end{document}